\documentclass[12pt]{article}
 \usepackage{amsmath}
\usepackage{amsfonts,amssymb}

\newcounter{aufgabe}

\newcounter{equiv}

\newcounter{afzg}

\newtheorem{aussage*}{$\!\!\!$}[section]
\newenvironment{aussage}%
{\setcounter{equation}%%@
{0}%
\begin{aussage*}}{\end{aussage*}}

\newenvironment{proof}%
{\setcounter{equation}%%@
{0}%
\begin{trivlist}\item[]{ Proof.}}%
{\hspace*{\fill}$\square$\end{trivlist}}

\newcommand{\Hom}{\mbox {\rm Hom\,}}

\newcommand{\Ke}{\mbox{\rm Ker\,}}
\newcommand{\Img}{\mbox{\rm Im\,}}

\newcommand{\End}{\mbox {\rm End}}

\newcommand{\op}{\oplus}

\newcommand{\sq}{\subseteq}

\newcommand{\Rej}{\rm Re \, }

\newcommand{\ba}{\begin{aussage} }
\newcommand{\ea}{\end{aussage} }

\newtheorem{thm}{Theorem}[section]
\newtheorem{cor}[thm]{Corollary}
\newtheorem{lem}[thm]{Lemma}
\newtheorem{prop}[thm]{Proposition}
\newtheorem{defn}[thm]{Definition}

\newtheorem{exam}[thm]{Example}

\title{ t-Dual Baer Modules and t-Lifting Modules }

\author {Tayyebeh Amouzegar$^1$, Derya Keskin T\"{u}t\"{u}nc\"{u}$^2$  and Yahya Talebi$^3$}
\date{}
\begin{document}
\maketitle
\begin{center}
 \small {
 $^{(1,3)}$ Department of Mathematics, Faculty of Science,
University of Mazandaran, Babolsar, Iran, e-mails:
t.amoozegar@yahoo.com,
talebi@umz.ac.ir\\

$^{(2)}$ Department of Mathematics, University of Hacettepe,  06800
Beytepe, Ankara, Turkey, e-mail: keskin@hacettepe.edu.tr }

\end{center}

\thanks{(11.10.2011: This work is a part of the PhD. thesis of the first author.)}

\begin{center}
\begin{abstract}
  We introduce the notions of t-lifting modules and t-dual Baer
  modules, which are generalizations of lifting modules. It is shown
  that an amply supplemented module $M$ is t-lifting if and only if
  $M$ is t-dual Baer and a t-$\mathcal{K}$-module. We also prove
  that, over a right perfect ring $R$, every noncosingular $R$-module is injective if and only if
  every $R$-module is t-dual Baer if and only if every $R$-module is
  t-lifting if and only if every injective $R$-module is t-lifting.
\end{abstract}
\end{center}
%\singlespace



 \section{Introduction}

Throughout this paper, $R$ will denote an arbitrary associative
ring with identity, $M$  a unitary  right $R$-module and
$S=\End(M)$ the ring of all $R$-endomorp-hisms of $M$. We will use
the notation $N\leq_e M$ to indicate that $N$ is essential in $M$
(i.e., $N\cap L\neq 0 \ \forall 0\neq L\leq M)$;
 $N\ll M$ means that $N$ is small
in $M$ (i.e. $\forall L \lneq M, L+N\neq M$). The notation $N
\leq^\op M$ denotes that $N$ is a direct summand of $M$.
  We also denote
     $D_S(N)=\{\phi \in S |\Img \phi \sq N\} $, for $N\sq M$.

Recall that an  $R$-module $M$ is  an \emph{extending}  module if
for every submodule  $A$
 of $M$ there exists a direct summand $B$ of $M$ such that
 $A\leq_e B$. Dually, a module $M$ is called a \emph{lifting }  module if,
for every  submodule $A$ of $M$ there exists a  direct summand $N$
of $M$ with $N\subseteq A$ and $A/N\ll M/N$. $M$ is lifting if and
only if $M$ is amply supplemented and every coclosed submodule of
$M$ is a direct summand (see \cite[22.3]{clvw}).

In \cite{tv}, Talebi and Vanaja defined $\overline Z(M)$  as
follows:
$$\overline Z(M) = \Rej (M, {\cal S}) = \bigcap \{\Ke(g) \, | \, g
\in \Hom (M, L), L \in {\cal S}\},$$ where $\cal S$ denotes the
class of all small modules. Note that any module is called {\it
small} if it is small in its injective hull.

 They called $M$ a {\em cosingular} ({\em noncosingular}) module if $\overline Z(M) = 0$
($\overline Z(M) = M$). Note that $\overline Z^2(M)$ is defined as
$\overline Z(\overline Z(M))$.

In \cite{ik}, Kaplansky introduced the concept of a Baer ring. A
ring $R$ is called  {\em right Baer }(resp.  \emph{left Baer}) if
the right (resp. left) annihilator of any nonempty subset
 of $R$ is generated  by an idempotent. Rizvi
and Roman introduced the concept of   Baer  modules in \cite{cr}.
According to \cite{cr}, $M$ is called a \emph{Baer } module if the
right annihilator in $M$ of any  left ideal  of $S$ is a direct
summand of $M$. In  \cite{kt}, Keskin-T\"{u}t\"{u}nc\"{u} and Tribak
introduced the concept of dual Baer
 modules. A module $M$ is called a \emph{dual Baer}
 module if for every right ideal
 $I$ of $S$, $\sum _{\phi \in I}\Img \phi$ is a direct
summand of $M$, equivalently, $D_S(N)$ is a direct summand of $M$
for every submodule $N$ of $M$. Asgari and Haghany introduced
t-extending and t-Baer modules in \cite{AH} as two generalizations
of extending modules. In this paper, motivated by this nice work,
we introduce t-lifting modules and t-dual Baer modules to
generalize lifting modules and obtain several dual results.

Let $M$ be a module and $A\leq M$. We say that $A$ is {\em
t-small} (written $A\ll_t M$) if for every submodule $B$ of $M$,
$\overline Z^2(M)\leq A+B$ implies that $\overline Z^2(M)\leq B$.
Some equivalent conditions for a t-small submodule are given in
Proposition \ref{2.2}. A submodule $C$ of a module $M$ is called
{\em t-coclosed} if $C/C'\ll_t M/C'$ implies that $C=C'$. We say
that a module $M$ is {\em t-lifting} if for every submodule $A$ of
$M$ there exists a direct summand $N$ of $M$ with $N\leq A$ and
$A/N\ll_t M/N$. In section 2, after giving some properties of
t-coclosed submodules, we get some equivalent statements for a
t-lifting module. We show that an amply supplemented module is
t-lifting if and only if every t-coclosed submodule is a direct
summand of $M$ if and only if $\overline Z^2(M)$ is a direct
summand of $M$ and $\overline Z^2(M)$ is lifting (Theorem
\ref{2.11}). Let $M$ be a module. We say that $M$ is a {\em t-dual
Baer module} if $I\overline Z^2(M)$ is a direct summand of $M$,
for every right ideal $I$ of $\End(M)$. We study t-dual Baer
modules and prove in section 3 that a module $M$ is t-dual Baer if
and only if $A\overline Z^2(M)$ is a direct summand of $M$ for
every subset $A$ of $\End(M)$ if and only if $\overline Z^2(M)$ is
a dual Baer direct summand of $M$ (Theorem \ref{3.2}). In
addition, a closed connection exists between t-lifting modules and
t-dual Baer modules; in fact, an amply supplemented module is
t-lifting if and only if it is t-dual Baer and a
t-$\mathcal{K}$-module (Theorem \ref{3.9}). Finally, we prove the
following:

 Let $R$ be a right perfect ring. Then the  following statements are equivalent:

$(1)$ Every noncosingular $R$-module is injective;

$(2)$ For every $R$-module $M$, $\overline Z^2(M)$ is a direct
summand of $M$ and $\overline Z^2(M)$ is  injective;

$(3)$  Every  $R$-module is t-dual Baer;

$(4)$ Every  $R$-module is t-lifting;

$(5)$ Every injective  $R$-module is t-lifting;

$(6)$ Every noncosingular  $R$-module is dual Baer and $\overline
Z^2(M)$ is a direct summand of $M$  for every $R$-module $M$;

$(7)$ Every noncosingular  $R$-module is lifting and $\overline
Z^2(M)$ is a direct summand of $M$ for every $R$-module $M$ (Theorem
\ref{3.12}).

For the undefined notions in this paper we refer to \cite{clvw}.

\section{t-coclosed submodules and t-lifting modules}

\begin{defn} {\rm A submodule $A$ of $M$ is called \emph{t-small} in $M$,
denoted by $A\ll_t M$, if for every submodule $B$ of $M$, $\overline
Z^2(M) \leq A+B$ implies that $\overline Z^2(M) \leq B$}.
\end{defn}
It is clear that if $A$ is a submodule of a noncosingular module
$M$, then $A$ is t-small in $M$ if and only if  $A$ is small in
$M$.

The concept of amply supplemented modules will be used
significantly in the paper. So we prefer to give its definition.
Any module $M$ is called {\em amply supplemented} if for any two
submodules $A$ and $B$ with $M=A+B$, $A$ contains a supplement of
$B$. Note that a submodule $X$ of any module $M$ is called a {\em
supplement} of any submodule $Y$ in $M$ if $M=X+Y$ and $X\cap Y$
is small in $X$.

\begin{prop} \label{2.2} Let  $M$ be an amply supplemented module and  $A$  a submodule of $M$. Then the  following statements are
equivalent:

$(1)$ $A$ is t-small in $M$.

$(2)$  $A\cap \overline Z^2(M) \ll \overline Z^2(M)$.

$(3)$  $A\cap \overline Z^2(M) \ll M$.

$(4)$  $\overline Z^2(A) =0$, namely, $\overline Z(A)$ is
cosingular.
\end{prop}

\begin{proof} $(1) \Rightarrow (2)$ Let $B\leq \overline Z^2(M)$ and
$(A\cap \overline Z^2(M)) +B=\overline Z^2(M)$. Then $\overline
Z^2(M)\subseteq A+B$. Since $A\ll_t M$, $\overline Z^2(M)\subseteq
B$.  Therefore $B=\overline Z^2(M)$ and so  $A\cap \overline Z^2(M)
\ll \overline Z^2(M)$.

$(2) \Rightarrow (3)$ It is clear.

$(3) \Rightarrow (4)$ $\overline Z^2(A)\subseteq A\cap \overline
Z^2(M)\ll M$, implies that $\overline Z^2(A)\ll M$. Hence $\overline
Z^2(A)$ is cosingular. On the other hand, by \cite[Theorem,
3.5]{tv}, $\overline Z^2(A)$ is noncosingular. Hence $\overline
Z^2(A)=0$.

$(4) \Rightarrow (1)$ Let $\overline Z^2(A)=0$ and $\overline
Z^2(M)\subseteq A+B$ for some submodule $B$ of $M$. By \cite[Theorem
3.5]{tv}, $\overline Z^2(M)=\overline Z^2(A+B)$ and $\overline
Z^2(A/(A\cap B))= (\overline Z^2(A)+(A\cap B))/(A\cap B)$. Since
$\overline Z^2(A)=0,$ $\overline Z^2(A/(A\cap B))=0$. Then
$\overline Z^2((A+B)/B)=0$. Again by \cite[Theorem 3.5]{tv},
$\overline 0=\overline Z^2((A+B)/B)=(\overline
Z^2(A+B)+B)/B=(\overline Z^2(M)+B)/B$ and so $\overline
Z^2(M)\subseteq B$.
\end{proof}

By Proposition \ref{2.2}, every small submodule of an amply
supplemented module $M$ and every supplement to $\overline Z^2(M)$
is t-small.

\begin{defn} {\rm A submodule $C$ of $M$ is called {\em t-coclosed}
in $M $ and denoted by $C\leq_{tcc} M$ if $C/C'\ll_t M/C'$ implies
that $C=C'$.}
\end{defn}
It is obvious that every t-coclosed submodule is coclosed in amply
supplemented modules and if $C$ is a submodule of a noncosingular
module $M$, then $ C$ is t-coclosed in $M$ if and only if $C$ is
coclosed in $M$.

\begin{lem} \label{2.5} Let $M$ be an amply supplemented module. Then:

$(1)$ If $C\leq_{tcc} M$, then $C\leq \overline Z^2(M)$.

$(2)$ $M\leq_{tcc} M$ if and only if $M$ is noncosingular.

$(3)$ If $A\subseteq C$ and $C\leq_{tcc} M$, then $C/A\leq_{tcc}
M/A$.

$(4)$ If $A\subseteq C$, $C/A\leq_{tcc} M/A$ and $A\leq_{tcc} M$,
then $C\leq_{tcc} M$.

$(5)$ If $A\subseteq C$ and $C$ is amply supplemented, then
$A\leq_{tcc} M
 \Leftrightarrow A\leq_{tcc} C$.
\end{lem}

\begin{proof} $(1)$ We have $C/(C\cap \overline Z^2(M))\cap \overline
Z^2(M/(C\cap \overline Z^2(M)))= C/(C\cap \overline Z^2(M)) \cap
\overline Z^2(M) /(C\cap \overline Z^2(M))=0\ll \overline
Z^2(M/(C\cap \overline Z^2(M))$. By Proposition \ref{2.2}, $C/(C\cap
\overline Z^2(M)) \ll_t M/(C\cap \overline Z^2(M))$. But
$C\leq_{tcc} M$, thus $C=C\cap \overline Z^2(M)$. Hence $C\leq
\overline Z^2(M)$.

$(2)$ Let $M\leq_{tcc} M$. By $(1)$, $M\subseteq \overline Z^2(M)$.
Then $M=\overline  Z^2(M)$. The converse is clear.

$(3)$ Let $C\leq_{tcc} M$. Let $\frac{C/A}{T/A}\ll_t
\frac{M/A}{T/A}$ for some submodule $T/A$ of $M/A$ with $T/A\leq
C/A$. Then $\overline Z^2(C/T)=0$ by Proposition \ref{2.2} and hence
$C/T\ll_t M/T$ by Proposition \ref{2.2} again. Thus $T=C$ since
$C\leq_{tcc} M$.

$(4)$ Let $C/T\ll_t M/T$ for some submodule $T$ of $M$ with $T\leq
C$. By Proposition \ref{2.2}, $\overline Z^2(C/T)=0$. Hence
$\overline Z^2(C)\leq T$ by \cite[Theorem 3.5]{tv}. Now, $\overline
Z^2(\frac{C}{C\cap (A+T)})=\frac{\overline Z^2(C)+[C\cap
(A+T)]}{C\cap (A+T)}=\frac{\overline Z^2(C)+A+(C\cap T)}{C\cap
(A+T)}=0$. Hence $\overline Z^2(\frac{C/A}{[C\cap(A+T)]/A})=0$. By
Proposition \ref{2.2}, $\frac{C/A}{[C\cap(A+T)]/A} \ll_t
\frac{M/A}{[C\cap(A+T)]/A}$. Then $C=C\cap (A+T)$ and so $C=A+T$.
Since $\overline Z^2(C/T)=0$, then $\overline Z^2(A/(A\cap T))=0$.
By Proposition \ref{2.2}, $A/(A\cap T)\ll_t M/(A\cap T)$. So,
$A=A\cap T$ and hence $A\subseteq T$. Thus $C=T$.

$(5)$ By Proposition \ref{2.2}.
\end{proof}

\begin{prop} \label{2.6} Let $C$ be a submodule of an amply
supplemented module $M$. Then the following are equivalent:

$(1)$ There exists a submodule $S$ such that $C$ is minimal with
respect to the property that $\overline Z^2(M) \subseteq C+S$.

$(2)$  $C$ is t-coclosed in $M$.

$(3)$ $C$ is contained in $\overline Z^2(M)$ and $C$ is a coclosed
submodule of $\overline Z^2(M)$.

$(4)$ $C$ is contained in $\overline Z^2(M)$ and $C$ is a coclosed
submodule of $M$.

$(5)$  $C$ is noncosingular.
\end{prop}

\begin{proof} $(1) \Rightarrow (2)$ Let $(1)$ hold and $C/C'\ll_t
M/C'$. Then $\overline Z^2(M) \subseteq C+C'+S$. Then $\overline
Z^2(M/C')= (\overline Z^2(M)+C')/C'\subseteq C/C'+(C'+S)/C'$. Since
$C/C'\ll_t M/C'$, $(\overline Z^2(M)+C')/C'\subseteq (C'+S)/C'$ and
so $\overline Z^2(M) \subseteq C'+S$. Hence $C=C'$.

$(2) \Rightarrow (3)$ By Lemma \ref{2.5}, $C$ is contained in
$\overline Z^2(M)$. Let $C/C'\ll \overline Z^2(M)/C'$. Then
$C/C'\cap \overline Z^2(M)/C'=C/C'\ll \overline Z^2(M)/C=\overline
Z^2(M/C)$. By Proposition \ref{2.2}, $C/C'\ll_t M/C'$. By
hypothesis, $C=C'$.

$(3) \Rightarrow (4)$ By \cite[Corollary 3.4]{tv},  $\overline
Z^2(M)$ is coclosed in $M$. By \cite[3.7(6)]{clvw},  $C$ is coclosed
in $M$.

$(4) \Rightarrow (3)$ By \cite[3.7(6)]{clvw}.

$(3) \Leftrightarrow (5)$ By \cite[Lemma 2.3(3) and Corollary
3.4]{tv}.

$(3) \Rightarrow (1)$ Let $C$ be  a coclosed submodule of $
\overline Z^2(M)$. Then $C$ is
 supplement in $\overline Z^2(M)$. Now, there exists a submodule $S$
 of $M$ such that $\overline Z^2(M) = C+S$ and $C$ is minimal with
$\overline Z^2(M) = C+S$. For any submodule $X$ of $M$ with
$X\subseteq C$, let $\overline Z^2(M) \subseteq X+S$. Then by
\cite[Theorem 3.5]{tv}, $\overline Z^2(M) = \overline Z^2(X+S)$.
Hence $C+S=\overline Z^2(M)=X+S$. By minimality of $C$ in $\overline
Z^2(M)$, $X=C$.
\end{proof}

Note that the conditions $(3)-(5)$ of Lemma \ref{2.5} are satisfied
from Proposition \ref{2.6}, as well.

\begin{cor} Let $M$ be an amply supplemented module. Then:

$(1)$ $\overline Z^2(M)$ is t-coclosed in $M$.

$(2) $ If $\phi$ is an endomorphism of $M$ and $C$ is a t-coclosed
submodule of $M$, then $\phi(C)$ is t-coclosed in $M$.
\end{cor}

\begin{proof} $(1)$ Since $\overline Z^2(M) $ is noncosingular,
$\overline Z^2(M)$ is t-coclosed in $M$ by Proposition \ref{2.6}.

$(2)$ Since $C$ is noncosingular, $\phi(C)$ is noncosingular. Thus
$\phi(C)$ is t-coclosed.
\end{proof}

The  sum of two coclosed submodules need not be coclosed (see
\cite[21.5]{clvw}), but this term is always true if we replace
coclosed with t-coclosed, as the following proposition shows.

\begin{cor} \label{2.8}  Let $M$ be an amply supplemented module.
Then an arbitrary sum of t-coclosed submodules of $M$ is t-coclosed.
\end{cor}

\begin{proof} Since arbitrary sum of noncosingular submodules is
noncosingular, it is clear.
\end{proof}

\begin{defn} {\rm A module $M$ is called \emph{t-lifting} if  every
submodule $A$ of $M$ contains a direct summand $B$ of $M$ such that
$A/B\ll_t M/B$.}
\end{defn}

The next result gives us several equivalent conditions for a
t-lifting amply supplemented module.

\begin{thm} \label{2.11} Let $M$ be an amply supplemented module.
Then the following are equivalent:

$(1)$ $M$ is t-lifting.

$(2)$ For every submodule $A$ of $M$, there exists a decomposition
$A=N\oplus N'$ such that $N$ is a direct summand of $M$ and
$N'\ll_t M$.

$(3)$  Every t-coclosed submodule of $M$ is a direct summand.

 $(4)$ For every submodule $A$ of $M$, $\overline Z^2(A)$ is a direct
summand of $M$.

$(5)$ For every coclosed submodule $A$ of $M$, $\overline Z^2(A)$
is a direct summand of $M$.

$(6)$ $\overline Z^2(M) $ is a direct summand of $M$ and
$\overline Z^2(M) $ is lifting.

$(7)$ Every submodule $A$ of $M$ which is contained in $\overline
Z^2(M) $, contains a direct summand $N$ of $M$ such that $A/N\ll
M/N$.
\end{thm}

\begin{proof} $(1) \Rightarrow (2)$ Let $A\leq M$. Then there exists
a decomposition $M=N\oplus L$ such that $A/N\ll_t M/N$. Then
$A=N\oplus (L\cap A)$. By Proposition \ref{2.2}, $\overline
Z^2(A/N)=0$ and so $\overline Z^2(L\cap A)=0$. Again by Proposition
\ref{2.2}, $L\cap A\ll_t M$.

 $(2) \Rightarrow (3)$ Let $C$ be a t-coclosed submodule of $M$. By
 assumption, $C=N\oplus N'$ such that $N\leq^\oplus M$ and $N'\ll_t
 M $. By Proposition \ref{2.2}, $\overline Z^2(N')=0$, thus
 $\overline Z^2(C/N)=0$. Again by Proposition \ref{2.2}, $C/N\ll_t
 M/N$. Since $C$ is t-coclosed, $C=N$ is a direct summand of $M$.

 $(3) \Rightarrow (4)$ Since $\overline Z^2(A)$ is noncosingular, by
 Proposition \ref{2.6}, $\overline Z^2(A)$ is t-coclosed in $M$ and
 so $\overline Z^2(A)$ is a direct summand of $M$.

$(4) \Rightarrow (5)$ It is clear.

$(5) \Rightarrow (6)$  Since $\overline Z^2(M)$ is coclosed in
$M$,
 $\overline Z^2(\overline Z^2(M))=\overline Z^2(M)$ is  a direct
summand of $M$. Now, let $C$ be a coclosed submodule of $\overline
Z^2(M)$. Thus, by \cite[Lemma 2.3]{tv}, $C$ is noncosingular. Hence
$\overline Z^2(C)=C$ and so $C$ is a direct summand of $M$.
Therefore  $C$ is a direct summand of $\overline Z^2(M)$.

$(6) \Rightarrow (7)$ Let $A\leq \overline Z^2(M)$. Then there
exists a direct summand $N$ of $\overline Z^2(M)$ such that $A/N\ll
\overline Z^2(M)/N$. Thus $A/N\ll  M/N$. It is clear that
$N\leq^\oplus M$.

$(7) \Rightarrow (1)$ Let $A\leq M$. By hypothesis, there exists a
direct summand $N$ of $M$ such that $(A\cap \overline Z^2(M))/N\ll
M/N$. By Proposition \ref{2.2}, $A/N\ll_t M/N$. Therefore $M$ is
t-lifting.
\end{proof}
It is clear that if $\overline Z^2(M)=0$, then $M$ is t-lifting,
where $M$ is amply supplemented. Every lifting module is t-lifting
since every t-coclosed submodule is coclosed in any amply
supplemented module.

\begin{exam}\label{first} {\rm $(1)$ It is well known that the
$\mathbb{Z}$-module $M=\mathbb{Z}/p\mathbb{Z}\oplus
\mathbb{Z}/p^2\mathbb{Z}$ is lifting, where $p$ is any prime. So $M$
is t-lifting.

$(2)$ It is well known that the $\mathbb{Z}$-module
$M=\mathbb{Z}/p\mathbb{Z}\oplus \mathbb{Z}/p^3\mathbb{Z}$ is not
lifting, but it is amply supplemented. Let $A$ be a t-coclosed
submodule of $M$. By Proposition \ref{2.5}, $A$ is noncosingular. On
the other hand, $A$ is cosingular since $M$ is cosingular. Thus
$A=0$, and hence it is a direct summand of $M$. Thus $M$ is
t-lifting by Theorem \ref{2.11}.}
\end{exam}

\begin{prop} Let $M$ be a t-lifting amply supplemented module. Then:

$(1)$ Every amply supplemented submodule of $M$ is t-lifting.

$(2)$ For every fully invariant submodule $L$ of $M$, $M/L$ is
t-lifting.
\end{prop}

\begin{proof} $(1)$ Let $A\leq M$ and $A$ be amply supplemented. Let
$L\leq A$. Since $M$ is t-lifting, there exists a direct summand $N$
of $M$ such that $N\subseteq L$ and $L/N\ll_t M/N$. Then $N$ is a
direct summand of $A$ and by Proposition \ref{2.2}, $L/N\ll_t A/N$.
 Hence $A$ is t-lifting.

 $(2)$ Let $L$ be a fully invariant submodule of $M$. Let $K/L\leq
 M/L$. Since $M$ is t-lifting, $M=N\oplus N'$, $N\subseteq K$ and
 $K/N\ll_t M/N$ for some submodule $N'$ of $M$. Note that $L=(N\cap
 L)\oplus (N'\cap L)= (N+L)\cap (N'+L)$ since $L$ is fully invariant in $M$. Hence $M/L=((N+L)/L)
 \oplus ((N'+L)/L)$. By Proposition \ref{2.2}, $\overline Z^2(K)\leq
 N$. Then $\overline Z^2(K/(N+L))=0$. Again by Proposition
 \ref{2.2}, $K/(N+L)\ll_t M/(N+L)$. Hence $M/L$ is t-lifting.
 \end{proof}

 \section{t-Dual Baer Modules}

\begin{defn} {\rm A module $M$ is said to be {\em t-dual Baer} if
$I( \overline Z^2(M))$ is a direct summand of $M$ for every right
ideal $I$ of $S$, where $S=\End(M)$.}
\end{defn}
It is clear that for a noncosingular module $M$, we have $M$ is
dual Baer if and only if it is t-dual Baer.

Recall that a module $M$ is said to have \emph{strongly summand sum
property} if the sum of every number of direct summand of $M$ is a
direct summand of $M$.

\begin{thm} \label{3.2}  Let $M$ be a module with $S=\End(M)$. Then the following are equivalent:

$(1)$ $M$ is t-dual Baer.

$(2)$ $\overline Z^2(M) $ is a direct summand of $M$ and
$\overline Z^2(M)$ is a dual Baer module.

$(3)$ $M$ has the strongly summand sum property for direct
summands which are contained in $\overline Z^2(M)$ and
$\phi(\overline Z^2(M))$ is a direct summand of $M$ for every
$\phi \in S$.

$(4)$ $\sum_{\phi \in A}\phi(\overline Z^2(M))$ is a direct summand
of $M$ for every subset $A$ of $S$.
\end{thm}

\begin{proof} $(1)\Rightarrow (2)$ Since $M$ is t-dual Baer,
$\overline Z^2(M)=S(\overline Z^2(M))$ is a direct summand of $M$.
Let $I$ be a right ideal of $\overline S=\End(\overline Z^2(M))$,
$A=\{i\phi \pi \ | \ \phi\in I \}$ where $\pi$ is the canonical
projection onto $\overline Z^2(M)$, $i$ is the inclusion map from
$\overline Z^2(M)$ to $M$ and $I'=AS$. It is clear that
$I(\overline Z^2(M))=I'(\overline Z^2(M))$. Since $M$ is t-dual
Baer, $I'\overline Z^2(M)$ is a direct summand of $M$. Thus
$I\overline Z^2(M)$ is a direct summand of $\overline Z^2(M)$.
Therefore $\overline Z^2(M)$ is dual Baer.

$(2)\Rightarrow (1)$  Let $I$ be a right ideal of $S$, $A'= \{
\pi'\phi|_{\overline Z^2(M)} \ : \ \phi\in I\} $ where $\pi'$ is
the canonical projection onto $\overline Z^2(M)$, $\overline
S=\End(\overline Z^2(M))$ and $I'=A'\overline S$. Since $\overline
Z^2(M)$ is dual Baer, $I'\overline Z^2(M)\leq^\oplus \overline
Z^2(M)$. It is clear that $I\overline Z^2(M)=I'\overline Z^2(M)$.
Since $\overline Z^2(M)\leq^\oplus M$, $I\overline
Z^2(M)\leq^\oplus M$.

$(1)\Rightarrow (3)$ Let $\phi \in S$. Since $\phi(\overline
Z^2(M))= \phi S(\overline Z^2(M))$ and $M$ is t-dual Baer,
$\phi(\overline Z^2(M))$ is a direct summand of $M$. Take
$e_i^2=e_i\in S, \ i\in \Lambda$ and $e_i(M)\subseteq \overline
Z^2(M)$. Let $I=\sum_{e_i\in \Lambda}e_i S$. Then $I(\overline
Z^2(M))=\sum_{\phi \in I}\phi(\overline Z^2(M))\leq \sum_{e_i\in
\Lambda} e_i M$. It is clear that $e_i(M)\subseteq \sum_{\phi \in
I}\phi(\overline Z^2(M)).$ Thus
 $ \sum_{e_i\in \Lambda} e_i M=\sum_{\phi \in I}\phi(\overline
Z^2(M))=I(\overline Z^2(M))\leq^\oplus M$ because $M$ is t-dual
Baer.

$(3)\Rightarrow (4)$ It is obvious, since $\phi(\overline
Z^2(M))\subseteq \overline Z^2(M)$ for every $\phi \in S$.

$(4)\Rightarrow (1)$ It is clear.
\end{proof}

Recall that a module $M$ is called a \emph{regular} module if every
cyclic submodule of $M$ is a direct summand of $M$.

\begin{cor} If $M$ has the strongly summand sum property for direct summands
which are contained in $\overline Z^2(M)$ and $M$ is regular, then
$M$ is t-dual Baer.
\end{cor}

\begin{proof} By Theorem \ref{3.2}, it suffices to show that $\phi(\overline
Z^2(M))$ is a direct summand of $M$ for every $\phi \in S$. Let
$\phi \in S$  and $N=\phi(\overline Z^2(M))$. Suppose that
$N=\sum_{x\in N}xR$. By hypothesis, $N$ is a direct summand of $M$.
\end{proof}

\begin{cor} If $M$ is regular t-dual Baer, then $\overline Z^2(M)$
is semisimple.
\end{cor}

\begin{proof} Let $N\leq \overline Z^2(M)$. Suppose that $N=\sum_{x\in
N}xR$. By Theorem \ref{3.2}, $N$ is a direct summand of $M$ and so
it is a direct summand of $\overline Z^2(M)$.
\end{proof}

Now we give a relation between the properties of dual Baer  and
t-dual Baer modules.

\begin{prop} \label{3.5} A module $M$ is dual Baer and $\overline Z^2(M)$ is a direct summand of $M$ if and
only if  $M$ is t-dual Baer and $\sum_{\phi \in A}\phi( M)/
\sum_{\phi \in A}\phi(\overline Z^2(M))$ is a direct summand of
$M/\sum_{\phi \in A}\phi(\overline Z^2(M))$ for every subset $A$ of
$S$.
\end{prop}

\begin{proof}  By Theorem \ref{3.2} and \cite[Corollary
2.5 and Theorem 2.1]{kt}.
\end{proof}

\begin{thm} \label{3.6} Every direct summand of a t-dual Baer module
is t-dual Baer.
\end{thm}

\begin{proof} Let $M=N\oplus N'$ and for every $i\in \Lambda $,
$K_i$ be a direct summand of $N$ such that $K_i\subseteq \overline
Z^2(N)$. Then $K_i\subseteq \overline Z^2(M)$ and since $M$ is
t-dual Baer, we have $\sum_{i\in \Lambda} K_i\leq^\oplus M$. Thus
$\sum_{i\in \Lambda} K_i\leq^\oplus N$. Let $f:N\rightarrow N$ be
a homomorphism. Consider the homomorphism $f\oplus 0_{N'}: N\oplus
N'\rightarrow N\oplus N'$ defined by $(f\oplus
0_{N'})(n+n')=f(n)$. Then $(f+0_{N'})(\overline
Z^2(M))=(f+0_{N'})(\overline Z^2(N)\oplus \overline
Z^2(N'))=f(\overline Z^2(N))$. As $M$ is t-dual Baer, $f(\overline
Z^2(N))\leq^\oplus M $ and hence it is a direct summand of $N$.
Therefore $N$ is t-dual Baer.
\end{proof}

Recall that a module $M$ is a \emph{$\mathcal{K}$-module} if for
every submodule $N$ of $M$, $D_S(N)=0$ implies that $N$ is small in
$M$.

Let $M$ be an $R$-module and $S= \End(M)$. For a submodule $N$ of
$M$ we denote $T_S(N)=\{\phi \in S: \phi(\overline
Z^2(M))\subseteq N\}.$

\begin{defn} {\rm  A module $M$ is called a \emph{t-$\mathcal{K}$-module} if
for every submodule $N$ of $M$, $T_S(N)=T_S(0)$ implies that $N$ is
t-small in $M$. Moreover, a module $M$ is called a \emph{strongly
t-$\mathcal{K}$-module} if for every submodule $N$ of $M$,
$T_S(N)=T_S(0)$ implies that $N$ is  small in $M$.}
\end{defn}

It is clear that every strongly t-$\mathcal{K}$-module is a
t-$\mathcal{K}$-module.  Obviously, for  noncosingular modules the
notions of $\mathcal{K}$-modules and t-$\mathcal{K}$-modules and
strongly t-$\mathcal{K}$-modules are equivalent.

\begin{prop} \label{3.8} Let $M$ be an amply supplemented module.
Then:

$(1)$ $M$ is a t-$\mathcal{K}$-module if and only if for every
submodule $N$ of $M$ which is contained in $\overline Z^2(M)$,
$T_S(N)=T_S(0)$ implies that $N$ is  small in $M$.

$(2)$ If $M$ is a t-$\mathcal{K}$-module, then $\overline Z^2(M)$ is
a $\mathcal{K}$-module.
\end{prop}

\begin{proof} $(1)$ The implication $(\Rightarrow)$ follows by Proposition
\ref{2.2}(3). For $(\Leftarrow)$, let $N$ be a submodule of $M$ and
$T_S(N)=T_S(0)$. Since $T_S(N\cap \overline Z^2(M))=T_S(N)=T_S(0)$,
by hypothesis, $N\cap \overline Z^2(M)$ is  small in $M$. Hence
$N\ll_t M$.

$(2)$ Let $\overline S=\End(\overline Z^2(M))$ and $N$ be a
submodule of $\overline Z^2(M)$ such that $D_{\overline S}(N)=0$.
Then $T_S(N)=T_S(0)$. For, let $\phi \in T_S(N)$, then $\overline
\phi=\phi|_{\overline Z^2(M)}:\overline Z^2(M)\rightarrow
\overline Z^2(M)$ is a homomorphism such that $\overline
\phi(\overline Z^2(M))\subseteq N$, thus $\overline \phi\in
D_{\overline S}(N)=0$ and so $\phi\in T_S(0)$; hence,
$T_S(N)=T_S(0)$. By hypothesis, $N$ is t-small in $M$. Therefore
$N\ll \overline Z^2(M)$ by Proposition \ref{2.2}.
\end{proof}

\begin{thm} \label{3.9}  Let $M$ be an amply supplemented module. Then the
following are equivalent:

$(1)$ $M$ is t-lifting.

$(2)$ $M$ is t-dual Baer and t-$\mathcal{K}$-module.

$(3)$ $M$ is t-dual Baer and $C=T_S(C)(\overline Z^2(M))$ for
every t-coclosed submodule $C$ of $M$.

$(4)$ $M$ is t-dual Baer and for every t-coclosed submodule $C$ of
$M$ if $T_S(C)=T_S(0)$, then $C=0$.
\end{thm}

\begin{proof} $(1)\Rightarrow (2)$ By Theorem \ref{2.11}, $\overline
Z^2(M)$ is a direct summand of $M$ and $\overline Z^2(M)$ is
lifting. By \cite[Theorem 2.14]{kt}, every noncosingular lifting
module is dual Baer and so $\overline Z^2(M)$ is dual Baer. By
Theorem \ref{3.2}, $M$ is t-dual Baer. Now, by proposition
\ref{3.8}, it suffices to show that if $N$  is a submodule of $M$
which is contained in $\overline Z^2(M)$, then $T_S(N)=T_S(0)$
implies that $N\ll M$. As $M$ is t-lifting, there exists a direct
summand $K$ of $M$ such that $N/K\ll_t M/K$. By Proposition
\ref{2.2}, $N/K\cap \overline Z^2(M/K)\ll M/K$. But $N/K\subseteq
(\overline Z^2(M)+K)/K=\overline Z^2(M/K)$, thus $N/K\ll M/K$. Let
$M=K\oplus K'$ and $K\neq 0$. Then $\overline Z^2(K)\neq 0$ since
if $\overline Z^2(K)=0$, then $0\neq K\subseteq N\subseteq
\overline Z^2(M)=\overline Z^2(K')\subseteq K'$. But $K\cap K'=0$,
contradiction. Now consider the canonical projection $\pi_K: M
\rightarrow K$. Then $\pi_K\in T_S(N)$ and $\pi_K\not \in T_S(0)$,
which is a contradiction. Therefore $K=0$ and so $N\ll M$.

$(1)\Rightarrow (3)$ By the proof of $(1)\Rightarrow (2)$, $M$ is
t-dual Baer. Let $C$ be a t-coclosed submodule of $M$. Obviously,
$T_S(C)(\overline Z^2(M))\subseteq C$. By hypothesis, $C$ is a
direct summand of $M$, say $M=C\oplus C'$. Consider the canonical
projection $\pi$ onto $C$. It is clear that $\pi\in T_S(C)$. By
Proposition \ref{2.6}, $C\subseteq \overline Z^2(M)$, thus
$C=\pi(C)\subseteq \pi(\overline Z^2(M))\subseteq T_S(C)(\overline
Z^2(M))$. Hence $C= T_S(C)(\overline Z^2(M))$.

$(2)\Rightarrow (4)$ Clear by Lemma \ref{2.5} and Propositions
\ref{2.2} and \ref{3.8}.

$(3)\Rightarrow (4)$ Let $C$ be a t-coclosed submodule of $M$ such
that $T_S(C)=T_S(0)$. By assumption, $C=T_S(C)(\overline
Z^2(M))=T_S(0)(\overline Z^2(M))=0$.

$(4)\Rightarrow (1)$ By Theorem \ref{2.11}, it suffices to show that
for any submodule $N$ of $M$ which is contained in $\overline
Z^2(M)$, there exists a direct summand $A$ of $M$ such that $N/A\ll
M/A$. Let $N$ be such a submodule of $M$. Since $M$ is t-dual Baer,
$eM=\sum_{\phi \in T_S(N)}\phi(\overline Z^2(M))=T_S(N)(\overline
Z^2(M))\subseteq N$ for some idempotent $e\in S$. If $N/eM$ is not
small in $M/eM$, then there exists a proper submodule $K/eM$ of
$M/eM$ with $eM\subseteq K$ such that $M/eM=K/eM+N/eM$. Restrict $N$
to a supplement $C$ of $K$ in $M$. $C$ is a coclosed submodule of
$M$ and $C\subseteq \overline Z^2(M)$, and so by Proposition
\ref{2.6}, $C$ is t-coclosed. Now we show that $T_S(C)=T_S(0)$. Let
$\phi \in T_S(C). $ Then $\phi(\overline Z^2(M))\subseteq C$, and so
$\phi(\overline Z^2(M))\subseteq N$, hence $\phi\in T_S(N)$. As
$eM=\sum_{\phi \in T_S(N)}\phi(\overline Z^2(M))$, we have
$\phi(\overline Z^2(M))\subseteq eM$. Thus $\phi(\overline
Z^2(M))\subseteq K$. Consequently, $\phi(\overline Z^2(M))\subseteq
K\cap C$. But $K\cap C\ll M$ implies that $\phi(\overline Z^2(M))\ll
M$. Hence $\phi(\overline Z^2(M))=0$. Hence $\phi\in T_S(0)$. Thus
$T_S(C)=T_S(0)$. By hypothesis $C=0$, and so $M=K$, which is a
contradiction. Therefore $N/eM\ll M/eM$.
\end{proof}

\begin{cor} \label{3.10} The following are equivalent for an amply supplemented
module $M$:

$(1)$ $M$ is noncosingular lifting.

$(2)$ $M$ is t-dual Baer and strongly t-$\mathcal{K}$-module.

$(3)$ $M$ is t-dual Baer and $C=T_S(C)(\overline Z^2(M))$ for
every coclosed submodule $C$.

$(4)$ $M$ is t-dual Baer and for any coclosed submodule $C$ of $M$,
if $T_S(C)=T_S(0)$, then $C=0$.
\end{cor}

\begin{proof} $(1)\Rightarrow (2)$  and $(1)\Rightarrow (3)$ By
Theorem \ref{3.9}.

$(2)\Rightarrow (4)$ This is clear.

$(3)\Rightarrow (4)$ Similar to the proof of Theorem \ref{3.9}
($(3)\Rightarrow (4)$).

$(4)\Rightarrow (1)$  By Theorem \ref{3.2}, $M=\overline
Z^2(M)\oplus K$ for some submodule $K$ of $M$ and $\overline Z^2(M)
$ is dual Baer. Clearly $K$ is closed submodule and $T_S(K)=T_S(0)$.
By $(4)$, $K=0$ and so $M=\overline Z^2(M)$. Hence  $M$ is
noncosingular.  By Theorem \ref{3.9}, $M$ is lifting.
\end{proof}

\begin{exam} {\rm (1) By Theorem \ref{3.9}, every lifting module
is t-dual Baer. But there exists t-dual Baer modules which are not
lifting. Consider the $\Bbb Z$-module $M=\Bbb Z/p\Bbb Z\oplus \Bbb
Z/p^3\Bbb Z$ in Example \ref{first}(2). It is amply supplemented
and t-lifting. By Theorem \ref{3.9}, it is t-dual Baer. But it is
not lifting.

(2) Let $R$ be a semiperfect ring which is not semisimple. Then
the right $R$-module $R_R$ is lifting by \cite[Corollary
4.42]{mm}. Hence it is t-lifting. Then by Theorem \ref{3.9}, $R_R$
is t-dual Baer. On the other hand, $R_R$ is not dual Baer by
\cite[Corollary 2.9]{kt}.

(3) If $R$ is a right $H$-ring, then every injective $R$-module is
t-lifting and t-dual Baer.}
\end{exam}

\begin{thm} \label{3.12} Let $R$ be a right perfect ring. Then the  following statements are equivalent:

$(1)$ Every noncosingular $R$-module is injective.

$(2)$ For every $R$-module $M$, $\overline Z^2(M)$ is a direct
summand of $M$ and $\overline Z^2(M)$ is  injective.

$(3)$  Every  $R$-module is t-dual Baer.

$(4)$ Every  $R$-module is t-lifting.

$(5)$ Every injective  $R$-module is t-lifting.

$(6)$ Every noncosingular  $R$-module is dual Baer and $\overline
Z^2(M)$ is a direct summand of $M$  for every $R$-module $M$.

$(7)$ Every noncosingular  $R$-module is lifting and $\overline
Z^2(M)$ is a direct summand of $M$ for every $R$-module $M$.
\end{thm}

\begin{proof} $(1) \Rightarrow (2)$  Since $\overline Z^2(M)$ is
noncosingular, by $(1)$, $\overline Z^2(M)$ is injective. Thus
$\overline Z^2(M)$ is a direct summand of $M$.

$(2) \Rightarrow (1)$ Clear.

$(2) \Rightarrow (3)$  Let $M$ be any $R$-module. By $(2)$,
$\overline Z^2(M)$ is a direct summand of $M$. Let $C$ be a
coclosed submodule of $\overline Z^2(M)$. By \cite[Lemma
2.3(2)]{tv}, $C$ is noncosingular. By $(2)$,  $C$ is injective,
and so it is a direct summand of $\overline Z^2(M)$. Consequently,
$\overline Z^2(M)$ is lifting. By \cite[Theorem 2.14]{kt},
$\overline Z^2(M)$ is dual Baer. Therefore $M$ is t-dual Baer by
Theorem \ref{3.2}.

$(4) \Rightarrow (5)$ Clear.

$(5) \Rightarrow (1)$  Let $M$ be a noncosingular  module and $E(M)$
be the injective hull of $M$. Since $M$ is noncosingular, $\overline
Z^2(M)=M$. By $(5)$, $E(M)$ is t-lifting. Then by  Theorem
\ref{2.11}(2), $\overline Z^2(M)\leq^\oplus E(M)$. Thus $M$ is
injective.

$(7) \Rightarrow (4)$ Let $M$ be any $R$-module. By $(7)$,
$\overline Z^2(M)$ is lifting and $\overline Z^2(M)\leq^\oplus M$.
Thus $M$ is t-lifting by Theorem \ref{2.11}.

$(3) \Rightarrow (6)$ Let $X$ be a noncosingular module. By $(3)$,
$X$ is t-dual Baer and hence it is dual Baer. Let $M$ be any
$R$-module. By $(3)$ and Theorem \ref{3.2}, $\overline
Z^2(M)\leq^\oplus M$.

$(3) \Rightarrow (4)$ Let $M$ be any $R$-module. Let $K\leq M$ and
define $\phi :  M\oplus K\rightarrow M\oplus K$ by $\phi(m,
k)=(k,0)$. Note that $M\oplus K$ is t-dual Baer by $(3)$. Then by
Theorem \ref{3.2}, $\phi(\overline Z^2(M\oplus K))=\phi(\overline
Z^2(M)\oplus \overline Z^2(K))=\overline Z^2(K)\oplus 0\leq^\oplus
 M\oplus K$. Thus $\overline Z^2(K)\leq^\oplus M$. By
Theorem \ref{2.11}, $M$ is t-lifting.

$(6) \Rightarrow (7)$ Let $M$ be a noncosingular $R$-module. Let $K$
be a coclosed submodule of $M$. By \cite[Lemma 2.3(3)]{tv}, $K$ is
noncosingular, and so $M\oplus K$ is noncosingular. Then by $(6)$,
$M\oplus K$ is dual Baer. Define $\phi :  M\oplus K\rightarrow
M\oplus K$ by $\phi(m, k)=(k,0)$. By \cite[Theorem 2.1]{kt},
$\phi(M\oplus K)=K\oplus 0\leq^\oplus M\oplus K$. Then $K\leq^\oplus
M$. So $M$ is lifting.
\end{proof}

\medskip \noindent
{\bf Acknowledgments}

This work has been done during a visit of the first author to the
second author in the Department of Mathematics, Hacettepe
University in 2011. She wishes to thank the Department of
Mathematics, Hacettepe University for their kind hospitality. The
first author also wishes to thank the Ministry of Science of Iran
for the support.

\end{document}